\documentclass[12pt, reqno]{amsart}
\usepackage{amsmath, amsthm, amscd, amsfonts, amssymb, graphicx, color}
\usepackage[bookmarksnumbered, colorlinks, plainpages]{hyperref}

\def\C{\mathbb C}

\def\R{\mathbb R}
\newfont{\hueca}{msbm10}

\newtheorem{theorem}{Theorem}[section]
\newtheorem{lemma}[theorem]{Lemma}
\newtheorem{proposition}[theorem]{Proposition}

\theoremstyle{definition}
\newtheorem{definition}[theorem]{Definition}
\newtheorem{example}[theorem]{Example}

\theoremstyle{remark}

\numberwithin{equation}{section}

\begin{document}
\setcounter{page}{1}

\title[Graded pseudo-$H$-rings]{Graded pseudo-$H$-rings}

\author[A.J. Calder\'on, A.  D\'{i}az, M. Haralampidou, J. M. S\'anchez]{Antonio Jes\'us
Calder\'on Mart\'{\i}n$^{1^*}$, Antonio D\'{i}az Ramos$^2$,
Marina Haralampidou$^3$ and Jos\'e Mar\'{i}a S\'anchez
Delgado$^4$}

\address{$^{1}$ Department of Mathematics\\University of C\'adiz\\
Campus de Puerto Real, 11510, C\'adiz, Spain.}
\email{\textcolor[rgb]{0.00,0.00,0.84}{ajesus.calderon@uca.es}}

\address{$^{2}$ Department of Algebra, Geometry and Topology\\University of M\'alaga\\
Campus de Teatinos, 29080, M\'alaga, Spain.}
\email{\textcolor[rgb]{0.00,0.00,0.84}{adiaz@agt.cie.uma.es}}

\address{$^{3}$ Department of Mathematics\\University of Athens\\
Panepistimioupolis, 15784, Athens, Greece.}
\email{\textcolor[rgb]{0.00,0.00,0.84}{mharalam@math.uoa.gr}}

\address{$^{4}$ Department of Mathematics\\University of C\'adiz\\
Campus de Puerto Real, 11510, C\'adiz, Spain.}
\email{\textcolor[rgb]{0.00,0.00,0.84}{txema.sanchez@uca.esr}}




\date{Received: xxxxxx; Revised: yyyyyy; Accepted: zzzzzz.
\newline \indent $^{*}$ Corresponding author}

\begin{abstract}
Consider a pseudo-$H$-space $E$ endowed with a separately
continuous biadditive associative multiplication which induces a
grading on $E$ with respect to an abelian group $G$. We call such
a space a graded pseudo-$H$-ring and we show that it has the form
$E = cl(U + \sum_j I_j)$ with $U$ a closed subspace of $E_1$ (the
summand associated to the unit element in $G$), and any $I_j$ runs
over a well described closed graded ideal of $E$, satisfying
$I_jI_k = 0$ if $j \neq k$. We also give a context in which graded
simplicity of $E$ is characterized. Moreover, the second
Wedderburn-type theorem is given for certain graded
pseudo-$H$-rings.
\end{abstract} \maketitle

\section{Introduction and Preliminaries}

In this paper we start considering the notion of a
\emph{pseudo-$H$-space}, that is, a real or complex vector space
$E$ equipped with a a family $(\langle  \cdot, \cdot
\rangle_{\alpha})_{\alpha \in I}$ of positive semi-definite
(pseudo)-inner products. We endow $E$ with the initial topology
with respect to the family of seminorms $(p_{\alpha})_{\alpha \in
I}$, where $p_{\alpha}(x):= \sqrt{\langle x,x \rangle_{\alpha}}$
for $\alpha\in I$ and $x \in E$. Thus $E$ becomes a locally convex
space (see \cite[p. 456, Definition 3.1]{Marina4}).

We assume that every \emph{pseudo-$H$-space} $E$ is Hausdorff and
complete. The former condition may be stated as follows: if $x\in
E$ satisfies $p_\alpha(x)=0$ for each $\alpha\in I$ then $x=0$.
The latter condition means that each Cauchy net in $E$ is
convergent. A \emph{pseudo-$H$-ring} is a pseudo-$H$-space endowed
with a biadditive associative multiplication separately
continuous, i.e., the endomorphisms of $E$ given by $x\mapsto xy$
and $x\mapsto yx$ are continuous for all $y \in E$. For any subset
$S$ of $E$ we shall denote its closure by $cl(S)$.

\begin{example}\label{ex1}
Let $I$ and $J$ be two arbitrary nonempty sets of  elements.
Consider the set $\mathbb{C}^{(I\times J) \times (I\times J)} $ of
all complex-valued functions $a$ on $(I \times J)\times (I \times
J) $ such that
\begin{itemize}
\item[(i)] $a((i,j),(l,m))=0$ when $j \neq m$ and
\item[(ii)]$\sum\limits_{i,k \in I; j \in J} |a((i,j),(k,j))|^2
\in \mathbb{R}_+$.
\end{itemize}

The latter, endowed with ``point-wise'' defined operations becomes
a vector space and an algebra with ``{\em matrix}'' multiplication
$$(ab)((i,j),(l,m))=\sum_{(k,s) \in I \times J} a((i,j),(k,s)) b((k,s),(l,m)),$$
for all $a, b \in \mathbb{C}^{(I\times J) \times (I\times J)}$.
Take a family of real numbers $(t_\alpha)_{\alpha \in \Lambda}$,
such that $t_\alpha \geq 1 $. For each $\alpha \in \Lambda $, the
mapping $ \langle  \cdot , \cdot \rangle_{\alpha} :
\mathbb{C}^{(I\times J) \times (I\times J)} \times
\mathbb{C}^{(I\times J) \times (I\times J)} \rightarrow
\mathbb{C}$ given by
$$\langle  a , b \rangle_{\alpha}= t_\alpha \sum\limits_{(i,j),(l,m) \in I \times J}
a((i,j),(l,m))\overline{b}((i,j),(l,m))$$ defines a pseudo-inner
product on $\mathbb{C}^{(I\times J) \times (I\times J)} $, where
``$-$'' denotes complex conjugation. Thus $E:=(
\mathbb{C}^{(I\times J) \times (I\times J)}, (\langle  \cdot ,
\cdot \rangle_{\alpha})_{\alpha \in \Lambda})$ becomes a locally
convex pseudo-$H$-ring.
\end{example}




Two elements $x,y$ in a pseudo-$H$-space $E$ are called
\emph{orthogonal} if $\langle x,y \rangle_{\alpha} = 0$ for all
$\alpha \in I.$ The \emph{orthogonal} set $S^{\perp}$ of a
non-empty subset $S$ in $E$ is defined by $$S^{\perp} := \{x \in
E: \langle x,y \rangle_{\alpha} = 0\textit{ for all } y \in S
\textit{ and all } \alpha \in I\}.$$ It is a closed linear
subspace of $E$. The symbol $\oplus^\bot$ shall denote {\it
orthogonal} direct sum, that is, a direct sum of mutually
orthogonal linear subspaces.



\begin{definition}\label{definition1}
Let $E$ be a  pseudo-$H$-ring, over ${\mathbb K}={\mathbb R}$ or
${\mathbb C}$,  and let $G$ be an
 abelian group. We say that $E$ is a {\it graded
pseudo-$H$-ring} (with respect to $G$) if
$$E = cl( \sideset{}{^\bot}\bigoplus\limits_{g \in G} E_g),$$
where $E_g$ is a closed linear subspace satisfying $E_gE_h \subset
E_{gh}$ (denoting by juxtaposition the product both in $E$ and
$G$), for any $g,h \in G$. We define the {\it support} of the
grading to be the set $\Sigma := \{g \in G \setminus \{1\} : E_g
\neq 0\}.$
\end{definition}


\medskip

Graded Hilbert spaces, and therefore graded classical
$H^*$-algebras,  and graded $l_2({\mathcal G})$ algebras, where
$\mathcal{G}$ is a compact topological group, are examples of
graded pseudo-$H$-rings (see \cite{Ambrose, Dix, Mallios}). Let us
also endow the family of pseudo-$H$-rings in Example \ref{ex1} of
different  gradings.

\begin{example}\label{ex2}
Consider the pseudo-$H$-ring $E= ( \mathbb{C}^{(I \times J)\times
(I \times J)}, (\langle  \cdot , \cdot \rangle_{\alpha})_{\alpha
\in \Lambda})$ of Example \ref{ex1}. Let  us fix an arbitrary
abelian group $G$. For any $((i,j), (k,j))$, $i,k\in I$ and $j \in
J$, denote $a_{((i,j), (k,j))}:(I \times J)\times (I \times J) \to
\mathbb{C}$ by
$$a_{((i,j), (k,j))}((l,m), (n,s)):=\left\{%
\begin{array}{ll}
    1, & \hbox{if $((l,m), (n,s))=((i,j), (k,j))$;} \\
    0, & \hbox{ otherwise.} \\
\end{array}%
\right.$$ the element units  in  $E$.     We have that any
function
$$\phi:I \times J \to G$$ gives rise to a $G$-grading on $E$ given by
$$\hbox{$\mathbb{C}a_{((i,j), (k,j))} \subset E_g$ if and only if $g=
\phi(i,j)^{-1}\phi(k,j)$}.$$ Indeed, taking into account
$a_{((i,j),(k,j))}a_{((m,l),(n,l))}=0$ for $(k,j)\neq (m,l)$, and
$$\phi(i,j)^{-1} \phi(k,j)\phi(k,j)^{-1} \phi(n,j) = \phi(i,j)^{-1}
\phi(n,j),$$ the above condition clearly defines the grading $E =
cl(\sideset{}{^\bot}\bigoplus\limits_{g \in G}E_g)$ with
\begin{equation}\label{br}
E_g=\sideset{}{^\bot} \bigoplus\mathbb{C}a_{((i,j), (k,j))},
\end{equation}
where the orthogonal direct sum is taken over all $i,k\in I; j\in
J $ with  $$\phi(i,j)^{-1}\phi(k,j)=g.$$
\end{example}

\bigskip

Let $E$ be a graded pseudo-$H$-ring. A {\it graded
pseudo-$H$-subring} $F$ of $E$ is a linear subspace with
$FF\subset F$ and which is decomposed as $F =
cl(\bigoplus^{\bot}\nolimits_{g \in G}F_g)$ where $F_g := F \cap
E_g$. A {\it graded ideal} $I$ of $E$ is a graded
pseudo-$H$-subring satisfying $IE\subset I$ and $EI \subset I$. A
graded pseudo-$H$-ring $E$ shall be called {\it graded simple} if
its product is nonzero and its only graded
ideals are $(0)$ and $E$. 
\medskip

In this work we study graded pseudo-$H$-rings $E$. In Section
\ref{section:Decompositions}, we give a particular decomposition
of $E$  as  $E = cl(U + \sum\limits_jI_j)$ with $U$ a closed
subspace of $E_1$ (the summand  associated to the unit element in
$G$), and any $I_j$ a well described closed graded ideal of $E$,
satisfying $I_jI_k = 0$ if $j \neq k$;. Then, in Section
\ref{section:simplecomponents} we  give a context in which graded
simplicity of $E$ is characterized. Moreover, a second
Wedderburn-type theorem is given for certain graded
pseudo-$H$-rings.

\smallskip

In the next lemma, by a {\em topological ring} we mean a
topological vector space which is a ring, such that the ring
multiplication is separately continuous.

\begin{lemma}\label{lema_util}
Let $E$ be a topological ring.
\begin{enumerate}
\item[(i)]\label{lema_utilclABclC} If $A,B,C$ are subsets of $E$
with $AB \subseteq C$, then $cl(A)cl(B) \subseteq cl(C)$.
\end{enumerate}
If, in addition $E$ is Hausdorff and complete, then the following
hold:
\begin{enumerate}
\item[(ii)]\label{lema_utilA+Bclosed} If $A$ and $B$ are
orthogonal closed subspaces of $E$, then $A \oplus B$ is closed.
\item[(iii)]\label{lema_utilcl(A+B)=A+cl(B)} If $A$ and $B$ are
orthogonal subspaces with $A$ closed, then $cl(A\oplus B) =
A\oplus cl(B)$.
\end{enumerate}
\end{lemma}

\begin{proof}

(i) See \cite[p. 6, Lemma 1.5]{Mallios}.

(ii) Let $A$ and $B$ be orthogonal closed subspaces and let
$(x_\lambda)_{\lambda \in \Lambda}$ be a net in $A\oplus B$
converging to $x_0$. For each $\lambda \in \Lambda$ we may write
$x_\lambda=a_\lambda+b_\lambda$ for unique elements $a_\lambda$
and $b_\lambda$ belonging to $A$ and $B$ respectively. For any
$\alpha\in I$, and as $A$ and $B$ are orthogonal, we have $\langle
x_\lambda,x_\lambda\rangle_\alpha=\langle
a_\lambda,a_\lambda\rangle_\alpha+\langle
b_\lambda,b_\lambda\rangle_\alpha.$ We deduce that $\langle
a_\lambda,a_\lambda\rangle_\alpha\leq \langle
x_\lambda,x_\lambda\rangle_\alpha$ and that $\langle
b_\lambda,b_\lambda\rangle_\alpha\leq \langle
x_\lambda,x_\lambda\rangle_\alpha$. Then $(a_\lambda)_{\lambda \in
\Lambda}$ and $(b_\lambda)_{\lambda \in \Lambda}$ are Cauchy nets
in $A$ and $B$ and, by the completeness of $E$, they converge to
$a_0\in A$ and $b_0\in B$ respectively. As $E$ is Hausdorff we
have $x_0=a_0+b_0$ and we are done.

(iii) We start by noticing that $A\oplus cl(B)\subseteq cl(A\oplus
B)$. By (ii) above, we get that $cl(A\oplus B)\subseteq A\oplus
cl(B)$. It is left to prove that the sum $A\oplus cl(B)$ is indeed
orthogonal. Take $a\in A$ and $c\in cl(B)$. Fix $\alpha\in I$. For
any $\epsilon \rangle0$ choose $b\in B$ such that
$p_\alpha(c-b)\langle \epsilon$. Then we have
$$\langle  a,c\rangle_\alpha=\langle  a,c-b+b\rangle_\alpha=\langle  a,c-b\rangle_\alpha\leq p_\alpha(a)p_\alpha(c-b)\langle \epsilon p_\alpha(a).$$
As this is true for each $\epsilon \rangle0$ we get that $\langle
a,c\rangle_\alpha=0$. So, by (ii), $A\oplus cl(B)$ is closed and
hence, $A\oplus cl(B)\subseteq A\oplus B \subseteq A\oplus cl(B)$.
Thus, the assertion follows.
\end{proof}


\section{Decompositions}\label{section:Decompositions}

From now on, $E$ denotes a graded pseudo-$H$-ring over $\R$ or
$\C$  and $$E = cl(\sideset{}{^\bot}\bigoplus\limits_{g\in G}E_g)
= E_1 \oplus cl(\sideset{}{^\bot}\bigoplus\limits_{g\in
\Sigma}E_g)$$ the corresponding grading, with support $\Sigma$,
and with respect to an abelian (multiplicative) group $G$.  Let us
denote by ${\Sigma}^{-1}:=\{h^{-1}: h \in \Sigma\} \subset G$.

\begin{definition}\label{def:connected}
Let $g, h$ be elements in $\Sigma$. We shall say that $g$ is {\em
connected} to $h$ if there exist $g_1,g_2...,g_n \in \Sigma \cup
{\Sigma}^{-1}$ such that
\begin{enumerate}
\item[(i)] $g_1 = g$. \item[(ii)] $\{g_1,g_1g_2,...,g_1g_2 \cdots
g_{n-1}\} \subset \Sigma \cup {\Sigma}^{-1}.$ \item[(iii)] $g_1g_2
\cdots g_{n-1}g_n \in \{h, h^{-1}\}$.
\end{enumerate}

We shall also say that $\{g_1,...,g_n\}$ is a {\em connection}
from $g$ to $h$.
\end{definition}


The next result shows that connectioness is an equivalence
relation.

\begin{proposition}\label{pro1}
Let $E$ be a graded  pseudo-$H$-ring with  support $\Sigma$. Then,
the relation $\sim$ in $\Sigma$, defined by $g \sim h$ if and only
if $g$ is connected to $h$, is an equivalence one.
\end{proposition}

\begin{proof}
Clearly  the set  $\{g\}$ is a connection from $g$ to itself
 and so
 the relation  is reflexive.

 If $g \sim h$ then
 there exists a connection
$\{g_1,g_2,...,g_n\} $ from $g$ to $h$:
$$\{g_1  g_2, g_1  g_2  g_3, ... , g_1
 g_2  \cdots  g_{n-1}\} \subset \Sigma \cup
{\Sigma}^{-1},$$ where $g_1  g_2 \cdots g_n \in \{ h, h^{-1}\}.$
Hence, we have two possibilities. In the first one $g_1  g_2
\cdots g_n = h,$ and in the second one $g_1  g_2  \cdots  g_n =
h^{-1}.$
 Now  observe that   the
set $$\{h,g_n^{-1},
 g_{n-1}^{-1},  ... , g_2^{-1}\}$$  gives us a connection from  $h$ to $g$ for the first possibility and   $\{h,g_n,
 g_{n-1},  ... , g_2\}$ for the second one. Hence  $\sim$ is symmetric.

Finally, suppose that $g \sim h$ and $h \sim k$, and
 write $\{g_1,g_2,...,g_n\}$ for a connection from
$g$ to $h$ and $\{h_1,h_2,...,h_m\}$ for a connection from $h$ to
$k$. If $h \notin \{ k, k^{-1}\}$, then $m \geq 1$ and so $\{g_1,
g_2,...,g_n,h_2,..., h_m\}$ (resp. $\{g_1,g_2,...,
g_n,h_2^{-1},...,h_m^{-1}\}$) is a connection from $g$ to $k$ if
$g_1g_2\cdots g_n=h$ (resp. $g_1g_2\cdots g_n=h^{-1}$). If $h \in
\{ k, k^{-1}\}$ then, $\{g_1,g_2,..., g_n\}$ is a connection from
$g$ to $k$. Therefore $g \sim k$ and this completes the assertion.
\end{proof}

By the above proposition   we can consider the quotient set
$$\Sigma / \sim=\{[g]: g \in \Sigma\},$$
where $[g]$ denotes the set of elements of $\Sigma$ which are
connected to $g$.
 By the definition of $\sim$, it is clear
that if $h \in [g]$ and $h^{-1} \in \Sigma$ then
 $h^{-1} \in [g]$.

\smallskip

Our next goal in this section is to associate a graded ideal
${E}_{[g]}$ to any $[g]$. Fix $g \in \Sigma$, we start by defining
the set ${E}_{1,[g]} \subset {E}_1$ as follows
 $${E}_{1,[g]}:= span_{\mathbb{K}}\{E_{h}E_{h^{-1}} : h \in
[g]\} \subset E_1.$$ Next, we define
\[
V_{[g]}:=\sideset{}{^\bot}\bigoplus \limits_{h \in [g]} E_{h}
\]
Finally, we denote by ${E}_{[g]}$ the following closed linear
subspace of $E$,
$${E}_{[g]} := cl(E_{1,[g]} \oplus^{\bot} V_{[g]}).$$

\begin{proposition}\label{pro2}
For any $g \in \Sigma$, the   linear subspace ${ E}_{[g]}$ is a
graded  pseudo $H$-subring of $E$.
\end{proposition}
\begin{proof}
We have
\begin{equation}\label{cero}
(E_{1,{[g]}} \oplus^{\bot} V_{[g]})(E_{1,{[g]}} \oplus^{\bot}
V_{[g]}) \subset E_{1,{[g]}}E_{1,{[g]}} + E_{1,{[g]}}V_{[g]} +
V_{[g]} E_{1,{[g]}} + V_{[g]}V_{[g]}.
\end{equation}

 Let us consider the last summand $V_{[g]}V_{[g]}$ in
(\ref{cero}). Given $h, k \in [g]$ such that $E_{h}E_{k} \neq 0$,
if $k = h^{-1}$ then, clearly $E_{h}E_{k} = E_{h}E_{{h}^{-1}}
\subset E_{1,{[g]}}$. Suppose that $k \neq h^{-1}$ and consider a
connection $\{g_1,...,g_n\}$ from $g$ to $h$. Since $E_{h}E_{k}
\neq 0$ implies $hk \in \Sigma$, we get that $\{g_1,...,g_n,k\}$
is a connection from $g$ to $hk$, in case $g_1 \cdots g_n = h$ and
$\{g_1 \cdots g_n, k^{-1}\}$, the respective one, in case $g_1
\cdots g_n = h^{-1}$. So $hk \in [g]$ and thus $E_{h}E_{k} \subset
E_{hk} \subset V_{[g]}$. Therefore, $(\sideset{}{^\bot}\bigoplus
\limits_{h \in {[g]}} E_{h})(\sideset{}{^\bot}\bigoplus \limits_{h
\in {[g]}} E_{h}) \subset E_{1,{[g]}} \oplus^{\bot} V_{[g]}$, that
is,
\begin{equation}\label{eq0.5}
V_{[g]}V_{[g]} \subset E_{1,{[g]}} \oplus^{\bot} V_{{[g]}}.
\end{equation}
Consider now the first summand $E_{1,{[g]}}E_{1,{[g]}}$ in
(\ref{cero}). By associativity, given $h, k \in [g]$, we have
$(E_{h}E_{h^{-1}})(E_{k}E_{k^{-1}}) \subset (E_{h}E_{h^{-1}})\cap
(E_{k}E_{k^{-1}}) \subset E_{1,{[g]}}.$ Hence,
\begin{equation}\label{eq0.75}
E_{1,{[g]}}E_{1,{[g]}} \subset E_{1,{[g]}}.
\end{equation}
Similarly, we show
\begin{equation}\label{eq0.8}
E_{1,{[g]}}V_{[g]} + V_{[g]}E_{1,{[g]}} \subset V_{[g]}.
\end{equation}
From the relations (\ref{cero}), (\ref{eq0.5}),(\ref{eq0.75}) and
(\ref{eq0.8}), we get $$(E_{1,{[g]}} \oplus^{\bot}
V_{[g]})(E_{1,{[g]}} \oplus^{\bot} V_{[g]}) \subset E_{1,{[g]}}
\oplus^{\bot} V_{[g]}.$$
 Finally, Lemma \ref{lema_util}-(i) completes the proof.
\end{proof}

\begin{lemma}\label{1111}
If $[g] \neq [h]$ for some $g, h \in \Sigma$ then ${E}_{[g]}
{E}_{[h]}=0$.
\end{lemma}
\begin{proof}
 We have $$(E_{1,{[g]}} \oplus^{\bot}
V_{[g]})(E_{1,{[h]}} \oplus^{\bot} V_{[h]}) \subset$$

\begin{equation}\label{cuatro}
E_{1,{[g]}}E_{1,{[h]}} + E_{1,{[g]}}V_{[h]} + V_{[g]}E_{1,{[h]}} +
V_{[g]}V_{[h]}.
\end{equation}
Consider the above last summand $V_{[g]}V_{[h]}$ and suppose there
exist $g_1 \in [g]$ and $h_1 \in [h]$ such that
$E_{g_1}E_{h_1}\neq 0$. Since $g_1 \neq h_1^{-1}$, then $g_1 h_1
\in \Sigma$. So $\{g_1,h_1,g_1^{-1}\}$ is a connection between
$g_1$ and $h_1$. By the transitivity of the connection relation,
we have $h \in [g]$, that is a contradiction. Hence
$E_{g_1}E_{h_1} = 0$ and thus
\begin{equation}\label{nueve}
V_{[g]}V_{[h]} = 0.
\end{equation}
Consider now the first summand $E_{1,[g]}E_{1,[h]}$ of
(\ref{cuatro}) and suppose there exist $g_1 \in [g]$ and $h_1 \in
[h]$ so that $(E_{g_1}E_{g_1^{-1}})(E_{h_1}E_{h_1^{-1}}) \neq 0$.
We have $E_{g_1}(E_{g_1^{-1}}E_{h_1})E_{h_1^{-1}} \neq 0$ and so
$E_{g_1^{-1}}E_{h_1} \neq 0,$ that contradicts (\ref{nueve}).
Hence $E_{1,{[g]}}E_{1,{[h]}}=0.$ Arguing in a similar way, we
also get
$$E_{1,{[g]}}V_{[h]}+V_{[g]}
E_{1,{[h]}} = 0.$$ From (\ref{cuatro}) we get
$$(E_{1,{[g]}} \oplus^{\bot} V_{[g]})(E_{1,{[h]}}
\oplus^{\bot} V_{[h]}) = 0.$$ Applying Lemma \ref{lema_util}, we
finally get $E_{[g]}E_{[h]} = 0$.
\end{proof}

\begin{theorem}\label{teo1}
In any pseudo-$H$-ring $E$  the following assertions hold.
\begin{enumerate}
\item[(i)] For any $g \in \Sigma$, the linear subspace
$$E_{[g]} = cl(E_{1,{[g]}} \oplus^{\bot} V_{{[g]}})$$ of
$E$ associated to $[g]$ is a graded ideal of $E$.

\item[(ii)] If $E$ is graded simple, then there exists a
connection between any two elements of  ${\Sigma}$ and $E_1 =
cl(span_{\mathbb{K}} \{E_gE_{g^{-1}}: g \in \Sigma\})$.
\end{enumerate}
\end{theorem}

\begin{proof}
(i) We first observe that by the grading (see Definition
\ref{definition1})
\begin{equation}\label{contencion1}
E_{h}E_{(h)^{-1}}E_1 \subset E_{h}E_{(h)^{-1}}
\end{equation}
and
\begin{equation}\label{contencion2}
E_{h}E_1 \subset E_{h}.
\end{equation}
Let us prove that $E_{[g]}E_1 \subset E_{[g]}$. From
(\ref{contencion1}), we obtain $E_{1,[g]}E_1 \subset E_{1,[g]} $,
and taking into account (\ref{contencion2}), we get $V_{[g]}E_1
\subset  V_{[g]}$. Therefore, $(E_{1,[g]} \oplus^{\bot}
V_{[g]})E_1 \subset E_{1,[g]} \oplus^{\bot} V_{[g]}$. Taking
closure, by Lemma \ref{lema_util}-(i) and the fact $E_1$ is
closed, we have
$$E_{[g]}E_1 \subset E_{[g]}.$$

Taking into account the above observation, Proposition \ref{pro2}
and Lemma \ref{1111}, we have $$E_{[g]}(E_1 \oplus^{\bot}
(\sideset{}{^\bot}\bigoplus\limits_{h \in [g]}E_{h}) \oplus^{\bot}
(\sideset{}{^\bot}\bigoplus\limits_{k \notin [g]}E_{k})) \subset
E_{[g]}.$$ Hence, Lemma \ref{lema_util} and the equality $$E =
cl(E_1 \oplus^{\bot} (\sideset{}{^\bot}\bigoplus\limits_{h \in
[g]} E_{h}) \oplus^{\bot} (\sideset{}{^\bot}\bigoplus\limits_{k
\notin [g]}E_{k}))$$ finally give $E_{[g]}E \subset E_{[g]}$.

In a similar way, we get $EE_{[g]} \subset E_{[g]}$ and so
$E_{[g]}$ is a graded ideal of $E$.

(ii) The graded simplicity of $E$ implies $E_{[g]} = E$. From
here, it is easy to get $[g] = \Sigma$ and $E_1 =
cl(span_{\mathbb{K}} \{E_gE_{g^{-1}}: g \in \Sigma\})$.
\end{proof}

\begin{theorem}\label{teo2}
Let $E$ be a pseudo-$H$-ring. Then for an orthogonal complement
$U$ of
$$cl(span_{\mathbb{K}}\{E_gE_{g^{-1}} : g \in \Sigma\})$$ in
$E_1$, we have
$$E = cl(U + \sum\limits_{[g] \in \Sigma/\sim}E_{[g]})$$ where
any $E_{[g]}$ is one of the (closed) graded ideals of $E$
described in Theorem \ref{teo1}-(i), satisfying $E_{[g]}E_{[h]} =
0$ if $[g] \neq [h].$
\end{theorem}

\begin{proof}
By Proposition \ref{pro1}, we can consider the quotient set
$\Sigma/\sim := \{[g] : g \in \Sigma\}.$ For any $[g] \in
\Sigma/\sim$ we know that  $ E_{[g]}$ is well defined and, by
Theorem \ref{teo1}-(i), it is a graded ideal of $E$. We also have
$E_1 \oplus^{\bot} (\bigoplus^{\bot}\limits_{g \in \Sigma}E_g) = U
+ \sum\limits_{[g] \in \Sigma/\sim}E_{[g]}$ and so
$$E = cl(U + \sum\limits_{[g] \in \Sigma/\sim}E_{[g]}).$$ By
applying Proposition \ref{pro2}-(ii), we get $E_{[g]}E_{[h]} = 0$
if $[g] \neq [h].$
\end{proof}

The linear subspace $E_1$ of $E$, associated to $1 \in G$, plays a
special role in any graded pseudo-$H$-ring $E = E_1 \oplus^{\bot}
cl(\sideset{}{^\bot}\bigoplus_{g\in \Sigma}E_g).$ Hence, in order
to obtain deeper structural descriptions of $E$ we have to
consider graded pseudo-$H$-rings in which $E_1$ and the
(pseudo-)inner products $\{\langle , \rangle_{\alpha}\}_{\alpha
\in I}$ of $E$ are compatible in a sense. From here, we introduce
the following notion motivated by the compatibility condition
between the inner product, the involution and the multiplication
which characterize a classical $H^*$-algebra (\cite{Ambrose}) and
its generalizations like Ambrose algebras (\cite{Marina1, Marina2,
Marina3}).

\begin{definition}\label{cohe}
We say that a graded pseudo-$H$-ring $(E, (\langle  \cdot, \cdot
\rangle_\alpha)_{\alpha \in I})$ has a {\it coherent 1-homogeneous
space} if $E_1 = cl(span_{\mathbb{K}} \{E_gE_{g^{-1}}: g \in
\Sigma\})$ and the following relation holds
$$\langle E_gE_{g^{-1}},E_{h}E_{h^{-1}} \rangle_{\alpha} = \langle E_g,E_{h}E_{h^{-1}}E_g \rangle_{\alpha}$$ for any $g, h \in
G$ and $\alpha \in I$.
\end{definition}

Graded classical $H^*$-algebras are examples of graded
pseudo-$H$-ring with coherent   1-homogeneous spaces. The graded
pseudo-$H$-rings in Example \ref{ex5} below are also examples of
graded pseudo-$H$-rings having coherent  1-homogeneous spaces.

\smallskip

\begin{theorem}\label{coro1}
Let $E$ be a pseudo-$H$-ring. If $E$ has a coherent 1-homogeneous
space,   then $$E = cl(\sideset{}{^\bot}\bigoplus_{[g] \in
\Sigma/\sim} E_{[g]}). $$ Namely, $E$ is the topological
orthogonal direct sum of the (closed) graded ideals given in
Theorem \ref{teo1}.
\end{theorem}

\begin{proof}
Taking into account Theorem \ref{teo2}, we clearly have from the
fact $$E_1 = cl(span_{\mathbb{K}} \{E_gE_{g^{-1}}: g \in
\Sigma\})$$ that $E = cl(\sum\limits_{[g] \in \Sigma/\sim}
E_{[g]})$. Since $E_1$ is coherent, and $E_{[g]}E_{[h]} = 0$, for
$[g] \neq [h]$ (Lemma \ref{1111}), we get
$$\langle E_{g_1}E_{g_1^{-1}},E_{g_2}E_{g_2^{-1}} \rangle_{\alpha} =
\langle E_{g_1},E_{g_2}(E_{g_2^{-1}}E_{g_1}) \rangle_{\alpha} =
0$$ for any $g_1 \in [g], g_2 \in [h]$ and $\alpha \in I$. Hence
the direct sum $\sideset{}{^\bot}\bigoplus\limits_{[g] \in
\Sigma/\sim} E_{1,[g]}$ is orthogonal. So, since $E_{[g]} =
cl(E_{1,[g]} \oplus^{\bot} (\sideset{}{^\bot}\bigoplus\limits_{h
\in [g]}E_{h}))$, we get the orthogonal direct character of the
sum of the ideals $E_{[g]}$, $[g] \in \Sigma/\sim$.
\end{proof}

\section{The graded simple components}\label{section:simplecomponents}

In this section, we study when the components in the
decompositions given in Theorems \ref{teo2} and \ref{coro1} are
graded simple. We begin by introducing the key notions of
$\Sigma$-multiplicativity and maximal length in the context of
graded pseudo-$H$-rings, in a similar way to that for graded
associative algebras, graded  Lie algebras, graded Poisson
algebras and so on. For these notions and examples see
\cite{Yogra, YoPo, Kochetov}.

\begin{definition}
It is said  that a graded pseudo-$H$-ring $E$ is of {\it maximal
length} if $E_1\neq 0$ and $\dim E_g = 1$ for any $g \in \Sigma$.
\end{definition}

\begin{definition}
We say that a graded pseudo-$H$-ring $E$ is {\it
$\Sigma$-multiplicative} if given $g \in \Sigma$ and  $h \in
\Sigma \cup \{1\}$ such that $gh \in \Sigma$, then
$E_gE_{h}+E_hE_{g} \neq 0$.
\end{definition}

We recall that $\Sigma$ is called {\it symmetric} when
$\Sigma=\Sigma^{-1}$ and that the {\it annihilator}  of $E$ is the
set  ${\mathcal Ann}(E) := \{v \in E : vE =0 \text{ and  } Ev =
0\}$. From now on $\Sigma$ will be supposed to be symmetric.

\begin{example}\label{ex5}
Consider the graded  pseudo-$H$-ring $E =
cl(\sideset{}{^\bot}\bigoplus\limits_{g \in G}E_g)$  where $$E= (
\mathbb{C}^{(I \times J) \times (I \times J)}, (\langle  \cdot ,
\cdot \rangle_{\alpha})_{\alpha \in \Lambda})$$ as in Example
\ref{ex2}. Take $I={\mathbb N}$, $J=\{1,2,...,r\}$ a finite set,
$G={\mathbb Q}^{\times}$, (the multiplicative rational group), and
a family of  $r$   sequences of prime natural numbers $\{x_{n,t}
\}_{n \in {\mathbb N}}$ where $t \in J$,  such that $x_{n,t} \neq
x_{m,s}$ when $(n,t) \neq (m,s)$. Define
$$\phi:{\mathbb N} \times J \to {\mathbb Q}^{\times}$$
$$\hspace{9mm}(n,p) \mapsto x_{n,p}$$ Taking
into account (\ref{br}) it is easy to verify that for any $q \in
{\mathbb Q}^{\times}$, $q \neq 1$,  either $E_q=0$ or $E_q=
\mathbb{C}a_{((n,t),(m,t))}$ for (unique) $n,m\in {\mathbb N}$ and
$t \in J$ such that $x_{n,t}^{-1}x_{m,t}=q$. In this case
$E_{q^{-1}}= \mathbb{C}a_{((m,t),(n,t))}$ and thus we get that $E$
is of maximal length and that its support is symmetric.

Since $$E_1=cl(\sideset{}{^\bot}\bigoplus\limits_{n \in {\mathbb
N};\hspace{1mm} t \in J}\mathbb{C}a_{((n,t),(n,t))})\neq 0$$ and
$a_{((n,t),(n,t))}=a_{((n,t),(m,t))}a_{((m,t),(n,t))}$ for any $m
\in {\mathbb N}$ with $m \neq n$, we also get that $E_1 =
cl(\sum\limits_{q \in \Sigma}E_qE_{q^{-1}})$ and
 so $E_1$ is coherent.

In order to verify that $E$ is $\Sigma$-multiplicative, take $q
\in \Sigma$ and   $p \in \Sigma $ such that $qp \in \Sigma$. By
the above we can write $q=x_{n,t}^{-1}x_{m,t}$ and
$p=x_{r,v}^{-1}x_{s,v}$. From here, if $qp \in \Sigma$ then either
$(m,t)=(r,v)$ or $(n,t)=(s,v)$. So, either
$qp=x_{n,t}^{-1}x_{s,t}$ and thus $E_q
E_p=\mathbb{C}a_{((n,t),(m,t))} \mathbb{C}a_{((m,t),(s,t))}=
\mathbb{C}a_{((n,t),(s,t))} \neq 0$ or $pq=x_{r,t}^{-1}x_{m,t}$
and we have $E_p E_q=\mathbb{C}a_{((r,t),(n,t))}
\mathbb{C}a_{((n,t),(m,t))}= \mathbb{C}a_{((r,t),(m,t))} \neq 0.$
If $p=1$, then clearly we have that $a_{((n,t),(m,t))}
a_{((m,t),(m,t))}=a_{((n,t),(m,t))}$ and so $E_q E_1 \neq 0$. Thus
$E$ is $\Sigma$-multiplicative.
\end{example}

\begin{theorem}\label{teo_final}
Let $E$ be a $\Sigma$-multiplicative graded pseudo-$H$-ring of
maximal length  and with ${\mathcal Ann}(E) = 0$. Then $E$ is
graded simple if and only if its support has all of its elements
connected and $E_1 = cl(span_{\mathbb{K}} \{E_gE_{g^{-1}}: g \in
\Sigma\})$.
\end{theorem}

\begin{proof}
For the first implication, see Theorem \ref{teo1}-(ii). To prove
the converse, consider $I = cl(\bigoplus^{\bot}_{g \in G}I_g)$,
where $I_g := I \cap E_g$, a nonzero graded ideal of $E$. We
denote by
$$\Sigma_I := \{g \in \Sigma : I_g \neq 0\}.$$
 By the maximal length
of the grading, if $g \in \Sigma_I$ then $0 \neq I_g=I \cap
E_g=E_g$ and so we can write $I = I_1 \oplus^{\bot}
cl(\bigoplus^{\bot}_{g \in \Sigma_I}E_g)$ where $I_1=I \cap E_1$.

Observe that $\Sigma_I \neq \emptyset$. Indeed, in the opposite
case $0 \neq I \subset E_1$ and then
$$I(\sideset{}{^\bot}\bigoplus\limits_{g \in \Sigma} E_g) \subset
(\sideset{}{^\bot}\bigoplus\limits_{g \in \Sigma} E_g) \cap E_1 =
0.$$ Therefore $I(\sideset{}{^\bot}\bigoplus\limits_{g \in
\Sigma}E_g) = 0$. In a similar way
$(\sideset{}{^\bot}\bigoplus\limits_{g \in \Sigma} E_g)I = 0$.
Hence, by Lemma \ref{lema_util}
\begin{equation}\label{eq40}
I(cl(\sideset{}{^\bot}\bigoplus \limits_{g \in \Sigma} E_g)) =
(cl(\sideset{}{^\bot}\bigoplus\limits_{g \in \Sigma}E_g))I = 0.
\end{equation}
Thus, the associativity of the product gives $I(E_gE_{g^{-1}}) +
(E_gE_{g^{-1}})I = 0$ and so, since $E_1 = cl(span_{\mathbb{K}}
\{E_gE_{g^{-1}}: g \in \Sigma\})$,  Lemma \ref{lema_util} implies
\begin{equation}\label{eq41}
IE_1 + E_1I = 0.
\end{equation}
From equations (\ref{eq40}) and (\ref{eq41}) we finally get $I
\subset {\mathcal Ann}(E)=0,$ a contradiction.

By the above we can take $g_0 \in \Sigma_I$, so that
\begin{equation}\label{eq_1}
0 \neq E_{g_0} \subset I.
\end{equation}
For  any $h \in \Sigma$ with $h \notin \{g_0,{{g_0}^{-1}}\}$.
Since $g_0$ and $h$ are connected, there is a connection
$\{g_1,g_2,...,g_r\}$ between them, such that
$$\hbox{$g_1 = g_0;$ \hspace{1mm} $g_1g_2,g_1g_2g_3,...,g_1g_2g_3 \cdots g_{r-1} \in \Sigma$ and
$g_1g_2g_3 \cdots g_r \in \{h,h^{-1}\}.$}$$

Consider $g_0 = g_1, g_2$ and $g_1g_2$. The
$\Sigma$-multiplicativity and maximal length of $E$ give $0 \neq
E_{g_0}E_{g_2} + E_{g_2}E_{g_0} = E_{g_0g_2}$. Thus, using
\eqref{eq_1}, we get $0 \neq E_{g_0g_2} \subset I.$

In a similar way, and employing the elements $g_0g_2, g_3$ and
$g_0g_2g_3$ we have $0 \neq E_{g_0g_2g_3} \subset I.$

Continuing this process on the connection $\{g_1,...,g_r\}$ we
obtain that $0 \neq E_{g_0g_2g_3 \cdots g_r} \subset I.$
Therefore, either $0 \neq E_{h} \subset I$ or $0 \neq E_{h^{-1}}
\subset I$ for any $h \in \Sigma.$ Since $E_1 =
cl(span_{\mathbb{K}} \{E_gE_{g^{-1}}: g \in \Sigma\})$  we
conclude
\begin{equation}\label{eq_3}
E_1 \subset I.
\end{equation}
Finally, given any $g \in \Sigma$, the $\Sigma$-multiplicativity
and maximal length of $E$ together with \eqref{eq_3} allow us to
assert that
\begin{equation}\label{eq_4}
0\neq E_g E_1+E_1 E_g=E_g \subset I.
 \end{equation}
From (\ref{eq_3}) and (\ref{eq_4}) we clearly get
 $I = E$. Hence,
$E$ is graded simple.
\end{proof}

We state now our main theorem that it is the second
Wedderburn-type theorem for certain graded pseudo-$H$-rings:

\begin{theorem} 
Let $E$ be a $\Sigma$-multiplicative graded pseudo-$H$-ring of
maximal length and with ${\mathcal Ann}(E) = 0$. If $E_1$ is
coherent then $E$ is the topological orthogonal direct sum of its
minimal (closed) graded ideals $E_{[g]}, g \in G$ . Moreover, each
$E_{[g]}$ is a graded simple, graded pseudo-$H$-ring, such that
the elements of its support are connected.
\end{theorem}

\begin{proof}
By Theorem \ref{coro1}, $E = cl(\oplus^{\bot}_{[g] \in \Sigma /
\sim}E_{[g]})$. Namely, $E$  is the topological orthogonal direct
sum of the ideals
$$E_{[g]} = cl(E_{1,[g]} \oplus^{\bot}
V_{[g]}) = cl(span_{\mathbb{K}}\{E_{h}E_{(h)^{-1}} : h \in [g]\})
\oplus^{\bot} cl(\oplus^{\bot}_{h \in [g]}E_{h}),$$ (see also
Theorem \ref{teo1} and the notation before Proposition
\ref{pro2}). We claim that the support, say $\sum_{E_{[g]}}$, of
$E_{[g]}, g\in G $, has all of its elements connected. Indeed,
since $[g] = [g^{-1}]$ and $E_{[g]}E_{[g]} \subset E_{[g]}$ (see
Proposition \ref{pro2}-1), we easily deduce that $[g]$ has all of
its elements $[g]$-connected (connected through elements in
$[g]$). Besides, the $\Sigma$-multiplicativity of $E$ implies that
of $E_{[g]}, g \in G$. Clearly $E_{[g]}$ is of maximal length.
Moreover, ${\mathcal Ann}(E_{[g]}) =\{0\}$ (the latter denotes the
annihilator of $E_{[g]}$ in itself), this is a consequence of the
fact that $E_{[g]}E_{[h]} = 0$ if $[g] \neq [h]$ (Theorem
\ref{teo2}), and ${\mathcal Ann}(E) = \{0\}$. An application of
Theorem \ref{teo_final} leads to the graded simpleness of
$E_{[g]}$. Thus, we easily get that any of the ideals $E_{[g]}$
are minimal, as well and this finishes the proof.
\end{proof}

\begin{example}\label{ex7}
Let us consider the pseudo-$H$-ring of Example \ref{ex5}. This is
$\Sigma$-multiplicative of maximal length with symmetric support
and it has a coherent $1$-homogeneous subspace. It is easy to
check that ${\mathcal Ann}(E)=0$.  Observe that given any $q,p \in
\Sigma$ with $p \notin \{q, q^{-1}\}$, we can write
$q=x_{n,t}^{-1}x_{m,t}$ with $n \neq m$ and
$p=x_{r,v}^{-1}x_{s,v}$ with $r \neq s$.

Suppose $v=t$.
By fixing some $u,v \in {\mathbb N}$  such that $u \notin \{n,m,r
\}$ and $v \notin \{m,r,s, u \}$ we get that
 the set
$$\{q,x_{u,t}^{-1}x_{n,t}, x_{m,t}^{-1}x_{v,t},x_{r,t}^{-1}x_{u,t},x_{v,t}^{-1}x_{s,t} \}$$ is a connection
from $q$ to $p$.

However, if $v\neq t$, and since $\Sigma=\{x_{n,t}^{-1}x_{m,t}:
n,m \in {\mathbb N}$ with $n \neq m$ and $t \in J\}$, there is not
any connection from $q$ to $p$. We have shown that the equivalence
classes in  $\Sigma / \sim $ are
$[x_{n,t}^{-1}x_{m,t}]=\{x_{r,t}^{-1}x_{s,t}: r,s \in {\mathbb N}$
with $r \neq s \}$ and by applying the results in this section we
can assert that, under the notation $$E_t= cl( (\sum\limits_{ n\in
{\mathbb N}}\mathbb{C}a_{((n,t),(n,t))})  \oplus
(\bigoplus\limits_{ n,m \in {\mathbb N}; n \neq
m}\mathbb{C}a_{((n,t),(m,t))}))$$ for any $t \in J$,  any $E_t$ is
a graded simple, graded pseudo-$H$-ring having  all of the
elements of its support connected. Moreover, $E$ decomposes as the
topological orthogonal direct sum of these family of minimal
graded ideals, namely:
$$E=cl(\sideset{}{^\bot}\bigoplus\limits_{t=1}^r E_t).$$
\end{example}

\medskip

{\bf Acknowledgment.} We would like to thank both of the referees
for the detailed reading of this work and for the suggestions
which have improved the final version of the same.

\bibliographystyle{amsplain}

\end{document}